\documentclass[12pt]{amsart}

\usepackage[a4paper,margin=2.5 cm]{geometry}
\address{\noindent
Pierre Youssef,\newline
Universit\'{e} Paris-Est\newline
Laboratoire d'Analyse et Math\'{e}matiques Appliqu\'ees (UMR 8050 CNRS)\newline
5, boulevard Descartes,
Champs sur Marne, \newline
77454 Marne-la-Vall\'{e}e, Cedex 2, France \newline
\texttt{e-mail: \small pierre.youssef@univ-mlv.fr}
}

\date{}
\usepackage{amsthm,hyperref,lmodern,mathrsfs}

\usepackage[latin1]{inputenc}
\usepackage[T1]{fontenc}
\usepackage{graphicx}
\usepackage[frenchb,english]{babel}
\usepackage{calc}
\usepackage{amsmath,amsthm}
\usepackage{graphicx}
\usepackage{url}
\usepackage{times, amssymb, amscd, mathrsfs, graphicx, color}
\usepackage{enumerate}

\newtheorem*{theoremeB}{Theorem B}
\newtheorem*{theoremeA}{Theorem A}
\newtheorem{theo}{Theorem}[section]
\newtheorem{prop}[theo]{Proposition}
\newtheorem{defi}[theo]{Definition}
\newtheorem{coro}[theo]{Corollary}
\newtheorem{lem}[theo]{Lemma}

\newtheorem{Rq}[theo]{Remark}
\newtheorem*{theoremeC}{Theorem C}

\title[]{A note on column subset selection}
\author{Pierre Youssef}
\begin{document}

\maketitle

\begin{abstract}
Given a matrix $U$, using a deterministic method, we 
extract a "large" submatrix of $\widetilde{U}$ (whose columns are obtained by 
normalizing those of $U$) 
and control its smallest and 
largest singular value. We apply this result to the study of contact points 
of the unit ball of a finite normed space with its maximal volume ellipsoid. 
We consider also the paving problem 
and give a deterministic algorithm to partition a matrix into almost isometric blocks 
recovering previous results of Bourgain-Tzafriri and Tropp. 
Finally, we partially answer a question raised by Naor 
about finding an algorithm in the spirit of Batson-Spielman-Srivastava's 
work to extract a "large" square submatrix of "small" norm.
\end{abstract}

\section{Introduction}

Let $U$ a $n\times m$ matrix, the stable rank of $U$ is given by
$
srank(U)=\frac{\Vert U\Vert_{\rm HS}^2}{\Vert U\Vert^2},
$
where $\Vert U\Vert_{\rm HS}^2={\rm Tr}(UU^t)$ denotes the Hilbert-Schmidt 
norm of $U$ and $\Vert U\Vert^2$ the square of the operator norm of $U$ seen as an 
operator from $l_2^m$ to $l_2^n$. 

Given $\sigma\subset \{1,...,m\}$, we denote $U_{\sigma}$ the restriction 
of $U$ to the columns with indices in $\sigma$ i.e. $U_{\sigma}=UP_{\sigma}^t$ 
where $P_{\sigma}: \mathbb{R}^m\longrightarrow \mathbb{R}^{\sigma}$ is the canonical 
coordinate projection. 

Denoting $\widetilde{U}$ the matrix 
whose columns are obtained by 
normalizing those of $U$, our aim is to extract 
almost  $srank(U)$ number of linearly independent columns of $\widetilde{U}$  
and estimate the smallest and the largest singular value of the restricted matrix. 

This problem is closely related to the restricted invertibility where only an estimate 
on the smallest singular value is needed. The restricted invertibility was first studied 
by Bourgain-Tzafriri \cite{MR890420} who proved the following:
\begin{theoremeA}
Given an $n\times n$ matrix $T$ whose columns are of norm one, there exists
$\sigma \subset \{1,\ldots ,n\}$ with $\vert \sigma\vert \geqslant d
\frac{n}{\Vert T \Vert_2^2}$ such that $ \Vert T_\sigma x\Vert_2
\geqslant c\Vert x\Vert_2$ for all $x\in \mathbb{R}^\sigma$, where
$d, c>0$ are absolute constants and $T_\sigma$ denotes the restriction 
of $T$ onto the columns in $\sigma$.
\end{theoremeA}

In \cite{MR1826503}, Vershynin extended this result for any decomposition of the identity, whereas 
the previous result was valid for the canonical decomposition. Moreover the size of the 
restriction depended on the Hilbert-Schmidt norm of the operator. Precisely, Vershynin 
proved the following:

\begin{theoremeB}
Let $Id = \sum_{j\leqslant m}
x_j x_j^t$ and let $T$ be a linear operator on $\ell_2^n$. For
any $\varepsilon \in (0,1)$ one can find $\sigma \subset \{1,\ldots
,m\}$ with
$$\vert \sigma\vert \geqslant (1-\varepsilon)\frac{\Vert T\Vert_{\rm
HS}^2}{\Vert T\Vert_2^2}$$ such that
$$ c(\varepsilon) \left(\sum_{j\in \sigma }
a_j^2\right)^{\frac{1}{2}}\leqslant
\left\Vert \displaystyle \sum_{j\in\sigma} a_j \frac{Tx_j}{\Vert Tx_j\Vert_2}\right\Vert_2
\leqslant C(\varepsilon) \left(\sum_{j\in \sigma }
a_j^2\right)^{\frac{1}{2}}$$ for all scalars $(a_j)$.
\end{theoremeB}
Vershynin  applied this result to the study of contact points. The normalization 
on the vectors $Tx_j$ is crucial for the applications and the dependence 
on $\varepsilon$ plays an important role.
Spielman-Srivastava \cite{MR2956233} generalized the restricted invertibility principle 
(Theorem A)  
for any decomposition of the identity without the normalizing factors appearing in 
Vershynin's result and improved the dependence on $\varepsilon$. In \cite{P.Y}, 
we unified the two previous results obtaining a good dependence 
on $\varepsilon$ for any normalizing factors. However this result deals only 
with the lower bound and doesn't give any information on the norm of the restricted 
matrix. Our aim here, is to improve Vershynin's result obtaining simultaneously 
a restricted invertibility principle and an estimate on the norm of the restricted matrix. 
Our proof uses ideas of the method of Batson-Spielman-Srivastava (\cite{MR2780071}, 
\cite{MR2941475}, see also \cite{P.Y} for related topics). \\

The main result of this paper is the following:

\begin{theo}\label{big block-colomn selection-smallest and largest singular value}
Let $U$ be an $n\times m$ matrix and denote by $\widetilde{U}$ the matrix 
whose columns are the columns of $U$ normalized.
For all $\varepsilon \in (0,1)$, there 
exists $\sigma\subset \{1,...,m\}$ of size
$$
\vert \sigma\vert \geqslant (1-\varepsilon)^2\frac{\Vert U\Vert_{\rm HS}^2}{\Vert U\Vert^2}
$$
such that 
$$
\frac{\varepsilon}{2-\varepsilon}\leqslant s_{\min}\left(\widetilde{U}_\sigma\right)
\leqslant s_{\max}\left(\widetilde{U}_\sigma\right)\leqslant \frac{2-\varepsilon}{\varepsilon}
$$
In other terms, for all $(a_j)_{j\in\sigma}$
$$
\frac{\varepsilon}{2-\varepsilon}\left(\sum_{j\in\sigma} a_j^2\right)^{\frac{1}{2}}\leqslant 
\left\Vert \sum_{j\in\sigma} a_j\frac{Ue_j}{\Vert Ue_j\Vert_2}\right\Vert_2\leqslant \frac{2-\varepsilon}{\varepsilon}\left(\sum_{j\in\sigma} a_j^2\right)^{\frac{1}{2}}.\\
$$
\\
\end{theo}

Note that the lower bound problem is the restricted invertibility problem treated in \cite{MR2956233} and \cite{P.Y} while the upper bound 
problem is related to the sparsification method and the Kashin-Tzafriri column selection theorem \cite{kashin-tzafriri} treated respectively in \cite{MR2780071} 
and \cite{P.Y}. Our idea is to merge the two algorithms together to get the two conclusions simultaneously. The heart of these methods 
is the study of the evolution of the eigenvalues of a matrix when perturbated by a rank one matrix.

In the regime where $\varepsilon$ is close to one, the previous result yields the following:
\begin{coro}\label{small block-colomn selection-smallest and largest singular value}
Let $U$ be an $n\times m$ matrix and denote by $\widetilde{U}$ the matrix 
whose columns are the columns of $U$ normalized.
For all $\varepsilon \in (0,1)$, there 
exists $\sigma\subset \{1,...,m\}$ of size
$$
\vert \sigma\vert \geqslant \frac{\varepsilon^2}{9}\cdot\frac{\Vert U\Vert_{\rm HS}^2}{\Vert U\Vert^2}
$$
such that 
$$
1-\varepsilon \leqslant s_{\min}\left(\widetilde{U}_\sigma\right)
\leqslant s_{\max}\left(\widetilde{U}_\sigma\right)\leqslant 1+\varepsilon
$$
In other terms, for all $(a_j)_{j\in\sigma}$
$$
(1-\varepsilon)\left(\sum_{j\in\sigma} a_j^2\right)^{\frac{1}{2}}\leqslant 
\left\Vert \sum_{j\in\sigma} a_j\frac{Ue_j}{\Vert Ue_j\Vert_2}\right\Vert_2\leqslant (1+\varepsilon)\left(\sum_{j\in\sigma} a_j^2\right)^{\frac{1}{2}}
$$
Where $(e_j)_{j\leqslant m}$ denotes the canonical basis of $\mathbb{R}^m$.
\end{coro}

This result is also related to the problem of column paving, that is, 
partitioning the columns into sets such that each of the corresponding 
restrictions has "good" bounds on the singular values, in particular 
such that the singular values are close to one. We will show how 
our Theorem allows us to recover a result of Tropp \cite{MR2807539} (and of 
Bourgain-Tzafriri \cite{MR890420})
dealing with this problem, using our deterministic method instead 
of the probabilistic methods used previously.\\

In a survey \cite{naor} on Batson-Spielman-Srivastava's sparsification theorem, 
Naor asked about giving a proof of another theorem of Bourgain-Tzafriri \cite{MR890420},
which is stronger than the restricted invertibility, using tools 
from Batson-Spielman-Srivastava's method. The theorem in question is the following:

\begin{theoremeC}[Bourgain-Tzafriri]
There is a universal constant $c>0$ such that for every $\varepsilon \in (0, 1)$ and
$n\in \mathbb{N}$, if $T$ is an $n\times n$ complex matrix satisfying $\langle
Te_i,e_i\rangle=0$ for all $i\in \{1,\ldots,n\}$ then there exists a
subset $\sigma\subseteq \{1,\ldots, n\}$ with $|\sigma|\geqslant
c\varepsilon^2 n$ and $\|P_\sigma TP_\sigma^*\|\leqslant \varepsilon\|T\|.$
\end{theoremeC}

Using our result, we will be able to give a deterministic algorithm to solve this 
problem for hermitian matrices.\\

The paper is organized as follows: in section $2$, we prove our main result, in 
section $3$ we give applications to the local theory of Banach spaces and compare it 
with Vershynin's result. In section $4$, we prove column paving results and finally 
in section $5$ we answer Naor's question.

\section{Proof of Theorem~\ref{big block-colomn selection-smallest and largest singular value}}

Note $k=\vert \sigma\vert = (1-\varepsilon)^2\frac{\Vert U\Vert_{\rm HS}^2}{\Vert U\Vert^2}$ and 
$$
A_k =\sum_{j\in\sigma}s_j\left(\widetilde{U}e_j\right)\cdot \left(\widetilde{U}e_j\right)^t,
$$
where $s_j$ are positive numbers which will be determined later.
 Since our aim is to find $\sigma$ 
such that the smallest singular value of $\widetilde{U}_\sigma$ is bounded away from zero and its largest 
one is upper bounded, it is equivalent to try to construct the matrix $A_k$ 
such that $A_k$ has $k$ eigenvalues bounded away from zero 
and bounded from above and to estimate the weights $s_j$. Our construction will be done step by step 
starting from $A_0=0$. So at the beginning, all the eigenvalues of $A_0$ are zero. At the first step we will try to find 
a vector $v$ among the columns of $\widetilde{U}$ and a weight $s$ such that $A_1=A_0+svv^t$ has one nonzero 
eigenvalue which have a lower and upper bound. Of course the first step is trivial, since for whatever column we choose 
the matrix $A_1$ will have one eigenvalue equal to $s$. At the second step, we will try to find a vector $v$ among the 
columns of $\widetilde{U}$ and a weight $s$ such that $A_2=A_1+svv^t$ has two nonzero eigenvalues for which we can 
update the lower and upper bound found in the first step. We will continue this procedure until we construct the matrix $A_k$.\\

For a symmetric matrix $A$ such that $b< \lambda_{\min} (A) \leqslant \lambda_{\max} (A) < u$, we define:
$$
\phi (A, b) = {\rm Tr}\left( U^t \left(A-b\cdot Id\right)^{-1}U\right) \quad \text{ and }\quad \psi (A, u)={\rm Tr}\left( U^t \left(u\cdot Id-A\right)^{-1}U\right)
$$
For $l\leqslant k$, we denote by $b_l$ the lower bound of the $l$ nonzero eigenvalues of $A_l$ and by $u_l$ the upper bound 
i.e. $A_l\prec u_l\cdot Id$ and has $l$ eigenvalues $>b_l$.
We also note
$$
\phi= \phi(A_0,b_0)= -\frac{\Vert U\Vert_{\rm HS}^2}{b_0} \quad \text{ and } \quad \psi=\psi(A_0,u_0)=\frac{\Vert U\Vert_{\rm HS}^2}{u_0},
$$ 
where $b_0$ and $u_0$ will be determined later.\\

As we said before, we want to control the evolution of the eigenvalues 
so we will make sure to choose a "good" vector so that our 
bounds $b_l$ and $u_l$ do not move too far. Precisely, we will fix 
this amount of change and denote it by $\delta$ for the lower bound and $\Delta$ for the upper bound i.e. 
at the next step the lower bound will be $b_{l+1}=b_l-\delta$ and the upper one $u_{l+1}=u_l+\Delta$. We will 
choose these two quantities as follows:
$$
\delta =(1-\varepsilon)\frac{b_0}{k}=\frac{b_0\Vert U\Vert^2}{(1-\varepsilon)\Vert U\Vert_{\rm HS}^2} \quad \text{ and } 
\quad \Delta= (1-\varepsilon)\frac{u_0}{k}=\frac{u_0\Vert U\Vert^2}{(1-\varepsilon)\Vert U\Vert_{\rm HS}^2}
$$
Our choice of $\delta$ is motivated by the fact that after $k$ steps we want the updated lower bound to remain positive 
but not too small due to some obstructions in the proof.
In this case the final lower bound will be 
$$
b_k=b_{k-1}-\delta=...=b_0-k\delta= \varepsilon b_0
$$
 
The choice of $\Delta$ is motivated by the fact that we don't want the upper bound to move too far from the initial one. 
The final upper bound will be
$$
u_k=u_{k-1}+\Delta=...=u_0+k\Delta=(2-\varepsilon)u_0
$$

\begin{defi}
We will say that a positive semidefinite matrix $A$ satisfies the $l$-requirement if the following properties
 are verified:

\begin{itemize}
\item $A\prec u_l \cdot Id$. 
\item $A$ has $l$ eigenvalues $>b_l$.
\item $\phi (A,b_l)\leqslant \phi$.
\item $\psi (A,u_l)\leqslant \psi$.
\end{itemize}
\end{defi}

In order to construct $A_{l+1}$ which has $l+1$ nonzero eigenvalues larger than $b_{l+1}$ and 
such that $\phi(A_{l+1},b_{l+1})\leqslant \phi(A_l,b_l)$, we may look at the algorithm used for the restricted 
invertibility problem (\cite{MR2956233}, \cite{P.Y}) and more precisely, at the condition needed on the vector $v$ 
to be chosen:

\begin{lem}
If $A_l$ has $l$ nonzero eigenvalues greater than $b_l$ and if for some vector $v$ and some positive scalar $s$ we have
\begin{equation}\label{eq-lowerbound}
G_l(v):=-\frac{v^t\left(A_l-b_{l+1}\cdot Id\right)^{-2}v}{\phi(A_l,b_l)-\phi(A_l,b_{l+1})}\cdot \Vert U\Vert^2 - v^t\left(A_l-b_{l+1}\cdot Id\right)^{-1}v\geqslant \frac{1}{s}.
\end{equation}
Then $A_{l+1}=A_l+svv^t$ has $l+1$ nonzero eigenvalues all greater than $b_{l+1}$ and $\phi(A_{l+1},b_{l+1})\leqslant \phi(A_l,b_l)$.
\end{lem}

\vskip 0.3cm

While in \cite{MR2956233} the restricted invertibility problem was unweighted, here we need a weighted 
version of it and this is why the previous lemma is stated like this. 
Now, in order to construct $A_{l+1}$ which has all its eigenvalues smaller than $u_{l+1}$ and 
such that  $\psi(A_{l+1},u_{l+1})\leqslant \psi(A_l,u_l)$ we may look at the algorithm used for the sparsification 
theorem \cite{MR2780071} or the Kashin-Tzafriri column selection theorem (see Theorem~4.2 \cite{P.Y}):

\begin{lem}
If $A_l\prec u_l.Id$ and if for some vector $v$ and some positive scalar $s$ we have
\begin{equation}\label{eq-upperbound}
F_l(v):=\frac{v^t\left(u_{l+1}\cdot Id-A_l\right)^{-2}v}{\psi(A_l,u_l)-\psi(A_l,u_{l+1})}\cdot \Vert U\Vert^2 +v^t\left(u_{l+1}\cdot Id-A_l\right)^{-1}v\leqslant \frac{1}{s}.
\end{equation}
Then denoting $A_{l+1}=A_l+svv^t$, we have $A_{l+1}\prec u_{l+1}.Id$ and $\psi(A_{l+1},u_{l+1})\leqslant \psi(A_l,u_l)$.
\end{lem}

The proofs of these two lemmas make use of the Sherman-Morrison formula. 

\vskip 0.3cm

For our problem, we will need to find a vector $v$ satisfying (\ref{eq-lowerbound}) and (\ref{eq-upperbound}) simultaneously. 
For that we need to merge these two conditions in one equation:
\begin{lem}\label{lem-lower-upper-bound}
If $A_l\prec u_l\cdot Id$ has $l$ eigenvalues greater than $b_l$ and if for some vector $v$ we have
\begin{equation}\label{eq-lower-upper-bound}
F_l(v)\leqslant G_l(v)
\end{equation}
Then taking any $s$ such that $F_l(v)\leqslant\frac{1}{s}\leqslant G_l(v)$, then $A_{l+1}=A_l+svv^t$ satisfies the $(l+1)$-requirement.
\end{lem}

\begin{Rq}\label{rq-relation on kernel}
Since $A_{l+1}$ has $l+1$ nonzero eigenvalues while $A_l$ had only $l$ nonzero eigenvalues, 
then the vector $v$ chosen has non-zero projection onto the kernel of $A_l$. Therefore one can 
see that ${\rm Ker}(A_{l+1})\subset {\rm Ker}(A_l)$ and ${\rm Dim\left[Ker(A_{l+1})\right]}={\rm Dim\left[ Ker(A_{l})\right]}-1$.
\end{Rq}

\begin{prop}\label{prop-iteration}
Let $A_l$ satisfy the $l$-requirement. If $b_0$ and $u_0$ satisfy
\begin{equation}\label{condition-u0-b0}
b_0\leqslant \frac{\varepsilon u_0}{2-\varepsilon}
\end{equation}
then there exists $i\leqslant m$ and a positive number $s_i$ 
such that $A_{l+1}=A_l+s_i\left(\widetilde{U}e_i\right)\cdot\left(\widetilde{U}e_i\right)^t$ 
satisfies the $(l+1)$-requirement.
\end{prop}

\begin{proof}
According to Lemma~\ref{lem-lower-upper-bound}, it is sufficient to find $i\leqslant m$ 
such that $F_l\left( \widetilde{U}e_i\right)\leqslant G_l\left(\widetilde{U}e_i\right)$ 
and then take $s_j$ such that
\begin{equation}\label{condition-weights}
F_l\left( \widetilde{U}e_i\right)\leqslant \frac{1}{s_j}\leqslant G_l\left(\widetilde{U}e_i\right)
\end{equation}
 Since $F_l$ and $G_l$ are quadratic 
forms, we are free to prove any weighted version of the previous inequality. For instance, it is equivalent to find $i\leqslant m$ such that $F_l\left(Ue_i\right)\leqslant G_l\left(Ue_i\right)$. 
For that, it is sufficient to prove
\begin{equation}\label{eq-sum-lower-upper-bound}
\sum_{j\leqslant m} F_l\left(Ue_j\right)\leqslant \sum_{j\leqslant m} G_l\left(Ue_j\right)
\end{equation}

Before estimating $\sum_{j\leqslant m} F_l\left(Ue_j\right)$, let us note that
\begin{align*}
\psi(A_l,u_l)-\psi(A_l,u_{l+1})&={\rm Tr}\left[U^t\left(u_{l}\cdot Id-A_l\right)^{-1}U\right]
-{\rm Tr}\left[U^t\left(u_{l+1}\cdot Id-A_l\right)^{-1}U\right]\\
&= \Delta {\rm Tr}\left[U^t\left(u_{l}\cdot Id-A_l\right)^{-1}\left(u_{l+1}\cdot Id-A_l\right)^{-1}U\right]\\
&\geqslant \Delta  {\rm Tr}\left[U^t\left(u_{l+1}\cdot Id-A_l\right)^{-2}U\right]
\end{align*}
Replacing this in $F_l$ we get
\begin{align*}
\sum_{j\leqslant m} F_l\left(Ue_j\right)&= \frac{\sum_{j\leqslant m}
e_j^tU^t\left(u_{l+1}\cdot Id-A_l\right)^{-2}Ue_j}{\psi(A_l,u_l)-\psi(A_l,u_{l+1})}\cdot \Vert U\Vert^2 
+\sum_{j\leqslant m}e_j^tU^t\left(u_{l+1}\cdot Id-A_l\right)^{-1}Ue_j\\
&= \frac{{\rm Tr}\left[U^t\left(u_{l+1}\cdot Id-A_l\right)^{-2}U\right]}{\psi(A_l,u_l)-\psi(A_l,u_{l+1})}\cdot \Vert U\Vert^2
+{\rm Tr}\left[U^t\left(u_{l+1}\cdot Id-A_l\right)^{-1}U\right]\\
&\leqslant \frac{\Vert U\Vert^2}{\Delta} +\psi
\end{align*}

Now we may estimate $\sum_{j\leqslant m}G_l\left(Ue_j\right)$ in a similar way to what is done in \cite{MR2956233} and \cite{P.Y}. 
We denote by $P_l$ the orthogonal projection 
onto the image of $A_l$ and $Q_l$ the orthogonal projection onto the kernel of $A_l$. Note that $\forall l\leqslant k$ we have 
the following
\begin{equation}\label{eq-relation on kernel}
b_l\leqslant \delta\frac{\Vert Q_lU\Vert_{\rm HS}^2}{\Vert U\Vert^2} 
\end{equation}
Since $Q_0=Id$, this fact is true at the beginning by our choice of $\delta$. Taking in account 
Remark~\ref{rq-relation on kernel}, at each step 
$\Vert Q_lU\Vert_{\rm HS}^2$ decreases by at most $\Vert U\Vert^2$ so that the 
right hand side of (\ref{eq-relation on kernel}) decreases by at most $\delta$. Since at 
each step we replace $b_l$ by $b_{l+1}$, (\ref{eq-relation on kernel}) remains true.\\
Since $Id=P_l+Q_l$ and $Q_l A_l = 0$, keeping in mind that $P_l, Q_l, A_l$ commute we can write
\begin{align*}
{\rm Tr}\left[U^t\left(A_l-b_{l+1}\cdot Id\right)^{-2}U\right]&=  
 {\rm Tr}\left[U^tP_l\left(A_l-b_{l+1}\cdot Id\right)^{-2}P_lU\right]+ 
 {\rm Tr}\left[U^tQ_l\left(A_l-b_{l+1}\cdot Id\right)^{-2}Q_lU\right]\\ 
&={\rm Tr}\left[U^tP_l\left(A_l-b_{l+1}\cdot Id\right)^{-2}P_lU\right]
  +\frac{\Vert Q_lU\Vert_{\rm HS}^2}{b_{l+1}^2}
  \end{align*}
Doing the same decomposition for $\phi_l(A_l)$ we get
\begin{align*}
\phi(A_l,b_l)&=  {\rm Tr}\left[U^tP_l\left(A_l-b_{l}\cdot Id\right)^{-1}P_lU\right]+ 
 {\rm Tr}\left[U^tQ_l\left(A_l-b_{l}\cdot Id\right)^{-1}Q_lU\right]\\
&={\rm Tr}\left[U^tP_l\left(A_l-b_{l}\cdot Id\right)^{-1}P_lU\right]- 
 \frac{\Vert Q_lU\Vert_{\rm HS}^2}{b_{l}}\\
&:=\phi^P(A_l,b_l) +\phi^Q(A_l,b_l)
\end{align*}
Denote $\Lambda_l = \phi(A_l,b_l) -\phi(A_l,b_{l+1})$ and $\Lambda_l^P$,  $\Lambda_l^Q$
 the corresponding decompositions onto the image part and the kernel part as above. As we did before, we have 
 $\Lambda_l =\Lambda_l^P+\Lambda_l^Q$ and
\begin{align*}
\Lambda_l^P&={\rm Tr}\left[U^tP_l\left(A_l-b_{l}\cdot Id\right)^{-1}P_lU\right]
 -{\rm Tr}\left[U^tP_l\left(A_l-b_{l+1}\cdot Id\right)^{-1}P_lU\right]\\
&= \delta {\rm Tr}\left[U^tP_l\left(A_l-b_{l}\cdot Id\right)^{-1}\left(A_l-b_{l+1}\cdot Id\right)^{-1}P_lU\right]\\
&\geqslant \delta {\rm Tr}\left[U^tP_l\left(A_l-b_{l+1}\cdot Id\right)^{-2}P_lU\right]\\
\end{align*}
Using (\ref{eq-relation on kernel}) we have
$$
\Lambda_l^Q = -\frac{\Vert Q_lU\Vert_{\rm HS}^2}{b_{l}}+\frac{\Vert Q_lU\Vert_{\rm HS}^2}{b_{l+1}}
=\frac{\delta \Vert Q_lU\Vert_{\rm HS}^2}{b_{l}b_{l+1}}\geqslant \frac{\Vert U\Vert^2}{b_{l+1}}
$$

\vskip 0.3cm

Looking at the previous information, we can write
\begin{align*}
\sum_{j\leqslant m} G_l\left( Ue_j\right)&=-\frac{
{\rm Tr}\left[U^t\left(A_l-b_{l+1}\cdot Id\right)^{-2}U\right]}{\Lambda_l}\cdot \Vert U\Vert^2 
-{\rm Tr}\left[U^t\left(A_l-b_{l+1}\cdot Id\right)^{-1}U\right]\\ 
&=-\frac{{\rm Tr}\left[U^tP_l\left(A_l-b_{l+1}\cdot Id\right)^{-2}P_lU\right]+
\frac{\Vert Q_lU\Vert_{\rm HS}^2}{b_{l+1}^2}}{\Lambda_l}\cdot \Vert U\Vert^2 
-\phi(A_l,b_{l+1})\\ 
&\geqslant -\frac{\frac{\Lambda_l^P}{\delta}+
\frac{\delta\Vert Q_lU\Vert_{\rm HS}^2}{b_{l}b_{l+1}}\left[ \frac{b_l}{\delta b_{l+1}}\right] }{\Lambda_l} 
\cdot \Vert U\Vert^2+\Lambda_l -\phi(A_l,b_l)\\
&\geqslant -\frac{\frac{\Lambda_l^P}{\delta}+
\Lambda_l^Q\left[ \frac{1}{\delta}+\frac{1}{ b_{l+1}}\right] }{\Lambda_l}
 \cdot \Vert U\Vert^2+\Lambda_l^Q -\phi\\
&\geqslant -\frac{\Vert U\Vert^2}{\delta}-\frac{\Vert U\Vert^2}{b_{l+1}}+\Lambda_l^Q -\phi\\
&\geqslant -\frac{\Vert U\Vert^2}{\delta}-\phi 
\end{align*}
Until now we have proven that 
$$
\sum_{j\leqslant m} G_l\left( Ue_j\right)\geqslant -\frac{\Vert U\Vert^2}{\delta}-\phi \quad 
\text{ and } \quad \sum_{j\leqslant m} F_l\left(Ue_j\right)\leqslant \frac{\Vert U\Vert^2}{\Delta} +\psi
$$
So in order to prove (\ref{eq-sum-lower-upper-bound}), it will be sufficient to verify
\begin{equation}\label{eq1-final condition}
\frac{\Vert U\Vert^2}{\Delta} +\psi\leqslant -\frac{\Vert U\Vert^2}{\delta}-\phi
\end{equation}

Replacing in (\ref{eq1-final condition}) the values of the corresponding parameters as chosen at the beginning, 
it is sufficient to prove
$$
\frac{(2-\varepsilon)\Vert U\Vert_{\rm HS}^2}{u_0}\leqslant \frac{\varepsilon \Vert U\Vert_{\rm HS}^2}{b_0}
$$
which is after rearrangement condition (\ref{condition-u0-b0}).
\end{proof}

Keeping in mind that $k = (1-\varepsilon)^2\frac{\Vert U\Vert_{\rm HS}^2}{\Vert U\Vert^2}$, 
we are ready to finish the construction of $A_k$. We may iterate Proposition~\ref{prop-iteration} starting with $A_0= 0$. 
Of course, $A_0$ satisfies the $0$-requirement so by the proposition we can find a column vector and 
a corresponding scalar to form $A_1$ satisfying the $1$-requirement. Once again we use the proposition 
to construct $A_2$ satisfying the $2$-requirement. We can continue with this procedure as long as the 
corresponding lower bound $b_l$ is positve (which is the case for $b_k$). So after $k$ steps we have constructed 
$A_k=\sum_{j\in\sigma}s_j\left(\widetilde{U}e_i\right)\cdot\left(\widetilde{U}e_i\right)^t$ satisfying the 
$k$-requirement which means that 
$$
A_k\prec u_k\cdot Id= (2-\varepsilon)u_0\cdot Id \text{ \ and \ $A_k$ has $k$ eigenvalues 
bigger than $b_k= \varepsilon b_0$.}
$$

Now it remains to estimate the weights $s_j$ chosen. In \cite{MR2780071}, one cannot get a lower bound for the weights. 
In our setting, the number of vectors chosen is always less than the dimension and we are able to estimate the weights by a trivial calculation 
using the normalization of the chosen vectors. 

\begin{lem}
For any $l\leqslant k$ and any unit vector $v$ we have
$$
G_l(v)\leqslant \frac{1}{\varepsilon b_0} \quad \text{ and } \quad F_l(v)\geqslant \frac{1}{(2-\varepsilon)u_0}
$$
\end{lem}

\begin{proof}
Write again $Id=P_l+Q_l$ and notice that \ $\frac{v^t\left(A_l-b_{l+1}\cdot Id\right)^{-2}v}{\phi(A_l,b_l)-\phi(A_l,b_{l+1})}$ \ and \ $v^tP_l\left(A_l-b_{l+1}\cdot Id\right)^{-1}P_lv$ are positive 
then we have
$$
G_l(v)\leqslant -v^t\left(A_l-b_{l+1}\cdot Id\right)^{-1}v
           \leqslant -v^tQ_l\left(A_l-b_{l+1}\cdot Id\right)^{-1}Q_lv
           \leqslant \frac{\Vert Q_lv\Vert_2^2}{b_{l+1}}
           \leqslant \frac{\Vert v\Vert_2^2}{b_k}
           \leqslant \frac{1}{\varepsilon b_0}
$$

Now since $\frac{v^t\left(u_{l+1}\cdot Id-A_l\right)^{-2}v}{\psi(A_l,u_l)-\psi(A_l,u_{l+1})}\geqslant 0$ then
$$
F_l(v)\geqslant v^t\left(u_{l+1}\cdot Id-A_l\right)^{-1}v\geqslant v^t\left(u_{k}\cdot Id\right)^{-1}v\geqslant  \frac{1}{(2-\varepsilon)u_0}
$$
\end{proof}

The weights $s_j$ that we have chosen satisfied (\ref{condition-weights}) and therefore by the previous lemma
$$
\forall i\leqslant k,\quad \varepsilon b_0 \leqslant s_j\leqslant (2-\varepsilon)u_0
$$

Back to our problem note that 
$$
\widetilde{U}_\sigma\widetilde{U}_\sigma^t = \sum_{j\in\sigma}\left(\widetilde{U}e_i\right)\cdot\left(\widetilde{U}e_i\right)^t
$$
and therefore 
$$
\frac{1}{(2-\varepsilon)u_0}A_k=\frac{1}{(2-\varepsilon)u_0}
\sum_{j\in\sigma}s_j\left(\widetilde{U}e_i\right)\cdot\left(\widetilde{U}e_i\right)^t\preceq
\widetilde{U}_\sigma\widetilde{U}_\sigma^t \preceq \frac{1}{\varepsilon b_0} 
\sum_{j\in\sigma}s_j\left(\widetilde{U}e_i\right)\cdot\left(\widetilde{U}e_i\right)^t 
=\frac{1}{\varepsilon b_0} A_k
$$

Transfering the properties of $A_k$, we deduce that 
$\widetilde{U}_\sigma\widetilde{U}_\sigma^t \preceq \frac{(2-\varepsilon) u_0}{\varepsilon b_0}Id$ 
and $\widetilde{U}_\sigma\widetilde{U}_\sigma^t$ has $k$ eigenvalues greater than $\frac{\varepsilon b_0}{(2-\varepsilon) u_0}$.

This means that
$$
\frac{\varepsilon b_0}{(2-\varepsilon) u_0} Id\preceq
\widetilde{U}_\sigma^t\widetilde{U}_\sigma \preceq \frac{(2-\varepsilon) u_0}{\varepsilon b_0}Id
$$
Taking $b_0=\frac{\varepsilon u_0}{2-\varepsilon}$ in order to satsify (\ref{condition-u0-b0}) 
we finish the proof of Theorem~\ref{big block-colomn selection-smallest and largest singular value}.

\section{Application to the local theory of Banach spaces}

As for the restricted invertibility principle where one can interpret 
the result as the invertibility of an operator on a decomposition of the identity, 
we will write the result in terms of a decomposition of the identity. 
This will be useful for applications to the local theory of Banach spaces 
since by John's theorem \cite{MR0030135} one can have 
a decomposition of the identity formed by contact points 
of the unit ball with its maximal volume ellipsoid.

\begin{prop}\label{prop-amelioration-vershynin}
Let $Id=\sum_{i\leqslant m} y_iy_i^t$ be a decomposition of the identity 
in $\mathbb{R}^n$ and $T$ be a linear 
operator on $l_2^n$.  For any $\varepsilon \in (0,1)$, there exists 
$\sigma\subset\{1,...,m\}$ such that
$$
\vert\sigma\vert\geqslant (1-\varepsilon)^2\frac{\Vert T\Vert_{\rm HS}^2}{\Vert T\Vert^2}
$$
and for all $(a_j)_{j\leqslant k}$,
$$
\frac{\varepsilon}{2-\varepsilon}\left(\sum_{j\leqslant k} a_j^2\right)^{\frac{1}{2}}
\leqslant \left\Vert\sum_{j\leqslant k} a_j\frac{Ty_j}{\Vert Ty_j\Vert_2}\right\Vert_2
\leqslant \frac{2-\varepsilon}{\varepsilon} \left(\sum_{j\leqslant k} a_j^2\right)^{\frac{1}{2}}.
$$
\end{prop}

\begin{proof}

Let $U$ be the $n\times m$ matrix whose columns are 
$Ty_j$. Therefore we can write
$$
UU^t= \sum_{j\leqslant m} \left(Ty_j\right)\cdot\left(Ty_j\right)^t 
=TT^t
$$ 
We deduce that $\Vert U\Vert_{\rm HS}=\Vert T\Vert_{\rm HS}$ and 
$\Vert U\Vert=\Vert T\Vert$.
Applying Theorem~\ref{big block-colomn selection-smallest and largest singular value} 
to $U$, we find $\sigma\subset\{1,...,m\}$ such that
$$
\vert \sigma\vert \geqslant (1-\varepsilon)^2\frac{\Vert T\Vert_{\rm HS}^2}{\Vert T\Vert^2}
$$
and for all $(a_j)_{j\in \sigma}$,
$$
\frac{\varepsilon}{2-\varepsilon}\left(\sum_{j\in\sigma} a_j^2\right)^{\frac{1}{2}}\leqslant 
\left\Vert \sum_{j\in\sigma} a_j\frac{Ue_j}{\Vert Ue_j\Vert_2}\right\Vert_2\leqslant \frac{2-\varepsilon}{\varepsilon}\left(\sum_{j\in\sigma} a_j^2\right)^{\frac{1}{2}}
$$
Noting that $\frac{Ue_j}{\Vert Ue_j\Vert_2}=\frac{Ty_j}{\Vert Ty_j\Vert_2}$, we finish the proof.

\end{proof}

This result improves the dependence on $\varepsilon$ 
in comparison with Vershynin's result  \cite{MR1826503}. While Vershynin proved
 that $(Ty_j)_{j\in\sigma}$ is $c(\varepsilon)$-equivalent to an 
orthogonal basis, the value of $c(\varepsilon)$ was of the order of $\varepsilon^{-c\log(\varepsilon)}$. 
Here our sequence is $(4\varepsilon^{-2})$-equivalent to an orthogonal basis. 
Using Corollary~\ref{small block-colomn selection-smallest and largest singular value}, one can write the 
previous proposition in the regime where $\varepsilon$ is close to one.

The previous result can be written in terms of contact points. 
Let us for instance write the case of $T=Id$. If $X =(\mathbb{R}^n, \Vert\cdot\Vert)$ 
where $\Vert\cdot\Vert$ is a norm on $\mathbb{R}^n$ such that $B_2^n$ is 
the ellipsoid of maximal volume contained in $B_X$ the unit ball of $X$, then by John's theorem 
one can get an identity decomposition formed by contact points of $B_X$ with $B_2^n$. 
Applying Proposition~\ref{prop-amelioration-vershynin} we get the following:

\begin{prop}\label{prop-contactpoints}
Let $X =(\mathbb{R}^n, \Vert\cdot\Vert)$ where $\Vert\cdot\Vert$ is a norm on $\mathbb{R}^n$ such that $B_2^n$ is 
the ellipsoid of maximal volume contained in $B_X$. For any $\varepsilon \in (0,1)$, 
there exists $x_1,...,x_k$ contact points of $B_X$ with $B_2^n$ 
such that 
$$
k\geqslant (1-\varepsilon)^2n
$$
and for all $(a_j)_{j\leqslant k}$,
$$
\frac{\varepsilon}{2-\varepsilon}\left(\sum_{j\leqslant k} a_j^2\right)^{\frac{1}{2}}
\leqslant \left\Vert\sum_{j\leqslant k} a_jx_j\right\Vert_2
\leqslant \frac{2-\varepsilon}{\varepsilon} \left(\sum_{j\leqslant k} a_j^2\right)^{\frac{1}{2}}.
$$
\end{prop} 

In other terms, we can find a system of almost $n$ contact points which is 
$(4\varepsilon^{-2})$-equivalent to an orthonormal basis.  Looking at the lower bound, this 
gives a proportional Dvoretzky-Rogers factorization \cite{MR1301496} with the best known 
dependence on $\varepsilon$.

If we are willing to give up on the fact 
of extracting a large number of contact points, we can have a system of contact points 
which is $(1+\varepsilon)$-equivalent to an orthonormal basis. For that we write the previous 
proposition in the regime where $\varepsilon$ is close to $1$.

\section{Column paving}

Extracting a large column submatrix reveals 
to be useful since the extracted matrix may have 
better properties. First results in this direction were 
given by Kashin in \cite{MR604848}, and others followed 
improving or dealing with different properties (see \cite{MR2780071}, \cite{MR890420}, \cite{kashin-tzafriri},
\cite{MR1001700},\cite{MR2807539}).
One can also be interested 
in partitioning the matrix into disjoint sets of columns 
such that each block has "good" properties. Obtaining 
a constant number of blocks (independent of the dimension) 
turns out to be a difficult problem and many 
conjectures concerning this were given previously (see \cite{MR2359423}).

The previous algorithms for 
extraction used probabilistic arguments and 
Grothendieck's factorization theorem. Here 
we propose a deterministic algorithm to achieve 
the extraction. We apply our main result 
iteratively in order to partition the matrix 
into blocks such that on each of them we have good 
estimates on the singular values.

\begin{defi}
Let $U$ a $n\times m$ matrix. We will say that $U$ is 
standardized if all its columns are of norm $1$.
\end{defi}

Note that when $U$ is standardized we have $\Vert U\Vert_{\rm HS}^2=m$ and $\Vert U\Vert\geqslant 1$. 
Applying Theorem~\ref{big block-colomn selection-smallest and largest singular value} to a standardized 
matrix, we get the following proposition:

\begin{prop}\label{prop-standardized-big-block}
Let $U$ a $n\times m$ standardized matrix. 
For $\varepsilon\in (0,1)$, there exists $\sigma\subset\{1,...,m\}$ 
with 
$$
\vert \sigma\vert\geqslant \frac{(1-\varepsilon)^2m}{\Vert U\Vert^2}
$$
such that 
$$
\frac{\varepsilon}{2-\varepsilon} \leqslant s_{\min}\left(U_{\sigma}\right)\leqslant s_{\max}\left(U_{\sigma}\right)
\leqslant \frac{2-\varepsilon}{\varepsilon}
$$
\end{prop}

In the regime where $\varepsilon$ is close to one, the previous 
proposition yields an almost isometric estimation.
Let us now use Proposition~\ref{prop-standardized-big-block} in order to 
partition the matrix into "good" blocks.
\begin{prop}\label{partition-big-block}
Let $U$ be an $n\times m$ standardized matrix. 
For $\varepsilon \in (0,1)$, there exists a partition of $\{1,...,m\}$ into 
$p$ sets $\sigma_1,...,\sigma_p$ such that
$$
p\leqslant \frac{\Vert U\Vert^2\log(m)}{(1-\varepsilon)^2}
$$
and for any $i\leqslant p$,
$$
\frac{\varepsilon}{2-\varepsilon} \leqslant s_{\min}\left(U_{\sigma_i}\right)\leqslant s_{\max}\left(U_{\sigma_i}\right)
\leqslant \frac{2-\varepsilon}{\varepsilon}
$$
\end{prop}

\begin{proof}
Apply Proposition~\ref{prop-standardized-big-block} 
to $U$ in order to get $\sigma_1$ verifying
 $$
\vert \sigma_1\vert\geqslant \frac{(1-\varepsilon)^2}{\Vert U\Vert^2}m
 $$ 
 such that 
$$
\frac{\varepsilon}{2-\varepsilon} \leqslant s_{\min}\left(U_{\sigma_1}\right)\leqslant s_{\max}\left(U_{\sigma_1}\right)
\leqslant \frac{2-\varepsilon}{\varepsilon}
$$
Now note that $U_{\sigma_1^c}$ is an $n\times \vert \sigma_1^c\vert$ 
standardized matrix and $\Vert U_{\sigma_1^c}\Vert\leqslant \Vert U\Vert$. 
A standard inductive argument finishes the proof.
\end{proof}

In the regime where $\varepsilon$ is close to one, 
the previous proposition yields a column partition with almost 
isometric blocks. This recovers, with a deterministic method, a result of Tropp (see Theorem 1.2 in \cite{MR2807539}), 
which follows results of Bourgain-Tzafriri \cite{MR890420}.
\begin{coro}\label{partition-small-block}
Let $U$ a $n\times m$ standardized matrix.
For $\varepsilon \in (0,1)$, there exists a partition of $\{1,...,m\}$ into 
$p$ sets $\sigma_1,...,\sigma_p$ such that
$$
p\leqslant \frac{9\Vert U\Vert^2\log(m)}{\varepsilon^2}
$$
and for any $i\leqslant p$,
$$
1-\varepsilon \leqslant s_{\min}\left(U_{\sigma_i}\right)\leqslant s_{\max}\left(U_{\sigma_i}\right)
\leqslant 1+\varepsilon
$$
\end{coro}

\vskip 0.3cm

The number of blocks here depends on the dimension. The challenging 
problem is to partition into a number of blocks which does not depend 
on the dimension. This would give a positive solution to the paving conjecture 
(see \cite{MR2359423} for related problems).

\section{Extracting square submatrix with small norm}

Let us first note that all previous sections trivially extend to complex matrices. 
In this section, we will show how using our main result 
we can answer Naor's question \cite{naor}: find an algorithm, using 
the Batson-Spielman-Srivastava's method \cite{MR2780071}, to prove 
Theorem~C \cite{MR890420}. However, we will be able to do this 
only for hermitian matrices.

\begin{prop}
Let $T$ an $n\times n$ Hermitian matrix with $0$ diagonal. 
For any $\varepsilon \in (0,1) $, there exists $\sigma\subset\{1,...,n\}$ 
of size
$$
\vert \sigma\vert \geqslant \frac{(\sqrt{2}-1)^4\varepsilon^2n}{2}
$$ 
such that 
$$
\left\Vert P_\sigma TP_\sigma^*\right\Vert\leqslant \varepsilon \Vert T\Vert
$$ 
\end{prop}

\begin{proof}
Denote $A=T+\Vert T\Vert\cdot Id$, then $A$ is a positive semidefinite Hermitian matrix 
so we may take $U=A^{\frac{1}{2}}$. First note that since $T$ has $0$ diagonal then
$$
\Vert Ue_i\Vert_2^2=\left<Ue_i,Ue_i\right> = \left<Ae_i,e_i\right>=\Vert T\Vert 
$$
Therefore $\widetilde{U}=\frac{U}{\Vert T\Vert^{\frac{1}{2}}}$ is a standardized matrix. Moreover 
$\Vert \widetilde{U}\Vert^2= 2$. \\

Denote $\alpha =(\sqrt{2}-1)^2$ and 
apply Proposition~\ref{prop-standardized-big-block} with $1-\alpha\varepsilon$ to find 
$\sigma\subset\{1,...,n\}$ of size $\frac{\alpha^2\varepsilon^2n}{2}$ such that 
$$
\frac{1-\alpha \varepsilon}{1+\alpha \varepsilon }\leqslant s_{\min}\left(\widetilde{U}_{\sigma}\right)
\leqslant s_{\max}\left(\widetilde{U}_{\sigma}\right)\leqslant 
\frac{1+\alpha \varepsilon }{1-\alpha \varepsilon}
$$
This means that 
$$
\left( \frac{1-\alpha \varepsilon}{1+\alpha \varepsilon }\right)^2\cdot Id_\sigma\preceq
\left(\widetilde{U}_\sigma\right)^*\cdot \left(\widetilde{U}_\sigma\right) 
\preceq \left( \frac{1+\alpha \varepsilon }{1-\alpha \varepsilon}\right)^2\cdot Id_\sigma
$$
Recall that $\widetilde{U}_\sigma= \widetilde{U}P_\sigma^*$ and 
$\widetilde{U}^*\cdot \widetilde{U}=\frac{A}{\Vert T\Vert}$. Therefore 
by the choice of $\alpha$
$$
(1-\varepsilon)\Vert T\Vert \cdot Id_\sigma\preceq
P_\sigma AP_\sigma^*\preceq
(1+\varepsilon)\Vert T\Vert\cdot Id_\sigma
$$
which after rearrangement gives
$$
-\varepsilon \Vert T\Vert\preceq P_\sigma TP_\sigma^*
\preceq \varepsilon \Vert T\Vert
$$
and finishes the proof.
\end{proof}

Iterating the previous result, we obtain by a deterministic method the strongest result on the paving 
problem which is due to Bourgain-Tzafriri (\cite{MR890420}, see also \cite{MR2379999}) that is every zero-diagonal matrix 
of size $n\times n$ can be paved with at most $O(\log (n))$ blocks. Once again, we are able to achieve this for 
Hermitian matrices.
\begin{prop}
Let $T$ be an $n\times n$ Hermitian matrix with $0$ diagonal. 
For any $\varepsilon \in (0,1) $, there exists a partition of $\{1,...,n\}$
into $k$ subsets $\sigma_1,..,\sigma_k$ such that 

$$
k\leqslant \frac{2 \log(n)}{(\sqrt{2}-1)^4\varepsilon^2}
$$ 
and for any $i\leqslant k$,
$$
\left\Vert P_{\sigma_i} TP_{\sigma_i}^*\right\Vert\leqslant \varepsilon \Vert T\Vert
$$ 

\end{prop}

\begin{proof}
As before denote $A=T+\Vert T\Vert\cdot Id$ and $U=A^{\frac{1}{2}}$. 
Note $\widetilde{U}=\frac{U}{\Vert T\Vert^{\frac{1}{2}}}$ the standardized matrix.
Applying  Corollary~\ref{partition-small-block}, we have a column partition for which we do 
on each block as we did in the previous proposition. The result follows easily.
\end{proof}


\vskip 0.5cm

\textbf{Acknowledgement :} I am grateful to my PhD advisor Olivier Guédon for his
constant encouragement and his careful review of this manuscript.

\nocite{*}
\bibliographystyle{abbrv}
\bibliography{bibliography-column-selection_new}

\end{document}